\documentclass[11pt]{amsart}

\usepackage[margin=1.5in]{geometry} 
\usepackage{amsmath,amsthm,amssymb, bbm, dsfont}
\usepackage{hyperref} 
\usepackage{adjustbox}
\usepackage[table,xcdraw]{xcolor}
\usepackage{listings}
\usepackage{amsmath}
\usepackage{amssymb}
\usepackage{bm}
\usepackage{graphicx}
\usepackage{psfrag}
\usepackage{color}
\usepackage{hyperref}
\hypersetup{colorlinks=true, linkcolor=blue, citecolor=magenta, urlcolor=magenta}
\usepackage{url}
\usepackage{algpseudocode}
\usepackage{fancyhdr}
\usepackage{mathtools}
\usepackage{tikz-cd}
\usepackage{xy}
\usepackage{colonequals}
\input xy
\xyoption{all}

\usepackage[textwidth=2cm,textsize=small,colorinlistoftodos]{todonotes}

\newtheorem{theorem}{Theorem}[section]
\newtheorem{lemma}[theorem]{Lemma}
\newtheorem{remark}[theorem]{Remark}
\numberwithin{equation}{section}

\def\Z{\mathbb{Z}}
\def\Q{\mathbb{Q}}
\def\F{\mathbb{F}}

\def\eps{\varepsilon}
\def\checkmark{\tikz\fill[scale=0.4](0,.35) -- (.25,0) -- (1,.7) -- (.25,.15) -- cycle;}
\usepackage[colorinlistoftodos]{todonotes}
\DeclareMathOperator{\rad}{rad}
\DeclareMathOperator{\Cl}{Cl}
\DeclareMathOperator{\GL}{GL}
\DeclareMathOperator{\Nm}{N}
\DeclareMathOperator{\N}{\mathcal{N}}
\DeclareMathOperator{\Gal}{Gal}
\DeclareMathOperator{\disc}{Disc}

\makeatletter
\def\pmod#1{\allowbreak\mkern5mu({\operator@font mod}\,\,#1)}

\makeatother

\DeclareMathAlphabet{\mathcal}{OMS}{cmsy}{m}{n}

\title{COUNTING $C_2 \wr S_4$ FIELDS WITH A POWER SAVING ERROR TERM}

\author{Sambhabi Bose}
\address{Department of Mathematics\\University of California, Berkeley\\Berkeley, CA 94720, US}
\email{sbose812@berkeley.edu}

\author{Kevin J. McGown}
\address{Department of Mathematics and Statistics\\California State University, Chico\\Chico, CA 95929, US}
\email{kmcgown@csuchico.edu}

\author{Ishan Panpaliya}
\address{Department of Mathematics\\Seattle University\\Seattle, WA 98122, US}
\email{ipanpaliya@seattleu.edu}

\author{Natalie Welling}
\address{Department of Mathematics\\University of Massachusetts Amherst\\Amherst, MA 01003, US}
\email{nwelling@umass.edu}

\author{Laney Williams}
\address{Department of Mathematics and Statistics\\California State University, Chico\\Chico, CA 95929, US}
\email{lcwilliams1@csuchico.edu}

\keywords{Malle's conjecture, counting fields, arithmetic statistics.}

\subjclass[2020]{Primary: 11R45, 11N45; Secondary: 11R29}

\date{\today}

\begin{document}

\begin{abstract}
    Let $\N_d(G,X)$ denote the number of degree $d$ extensions of $\Q$ with Galois closure $G$ and $|\Delta_K|\leq X$.
    Malle's conjecture predicts an asymptotic of the form $\N_d(G,X)\sim CX^{\alpha}(\log X)^\beta$.
    Previously, Kl\"uners proved Malle's conjecture for $G=C_2 \wr S_4$.  His proof gives a power savings of $O(X^{7/8})$.  We improve Kl\"uners' result
    by establishing a stronger power saving error term for the count of such fields.  Specifically, we show
    $\N_8(C_2\wr S_4,X)=CX+O(X^{3/4-1/30})$.
    Additionally, we obtain new bounds on $\N_8(G,X)$ for the groups
    $S_4$, $C_2^3 \rtimes S_4$, $\GL_2 (\F_3)$, and $Q_8\rtimes S_4$ as permutation subgroups of $S_8$.

\end{abstract}

\maketitle

\section{Introduction}

Counting number fields is a central problem in arithmetic statistics.
Given an integer $d\geq 2$ and a transitive permutation subgroup $G$ of $S_d$,
one can consider the collection of (isomorphism classes of) 
degree $d$ extensions $L$ over a fixed base field $K$
with the property that the Galois closure
$\Gal(\widetilde{L}/K)$ is permutation isomorphic to $G$.  One defines the counting function
$$
\N_{d,K}(G,X) \colonequals \#\{L\mid[L:K]=d,\,\Gal(\widetilde{L}/K)\cong G,\,\Nm(\Delta_{L/K})\leq X\}
\,.
$$
When $K=\Q$, we drop $K$ from the notation; that is, $\N_{d}(G,X)\colonequals \N_{d,\Q}(G,X)$.
Similarly, when the Galois group $G$ is not specified, we write $\N_{d,K}(X)\colonequals\sum_G\N_{d,K}(G,X)$ and $\N_d(X)\colonequals \N_{d,\Q}(X)$.
A folklore conjecture, often attributed to Linnik, speculates that
$\N_{d}(X)\sim C_d X$ for any degree $d$.
This conjecture is known for $d=2$, essentially
by results of Gauss, for $d=3$ by the seminal work of Davenport--Heilbronn, and for $d=4,5$ by the
groundbreaking
results of Bhargava (see~\cite{35b3f195-a21d-3d48-b94d-d206914c380e, MR2183288, MR2745272}).
However, this conjecture is not known for any value of $d>5$.

Malle conjectured (see~\cite{MR1884706}) that for a predetermined constant
$\alpha\colonequals \alpha(G)$, the relation
\[
X^\alpha \ll_{K, G}  \N_{d,K}(G, X) \ll_{K, G}  X^{\alpha + \eps }
\]
holds for any $\eps > 0$.
It is important to note that the constant $\alpha$ depends on the permutation group $G$
(as it is embedded into $S_d$), and not just on $G$ as an abstract group.
One defines
\[
\alpha^{-1} \colonequals \min \{\text{ind}(g) \colon g \in G \setminus \{e_G\}\},
\]
where the \emph{index} of $g \in G$, denoted by $\text{ind}(g)$, is defined as
$d$ minus the number of $g$-orbits of $\{1,\dots,n\}$.  For example, if $G$ contains
a transposition, then $\alpha=1$.

Later (see~\cite{MR2068887}), Malle also posited that,
for constants $C\colonequals C(K,G)$ and $\beta\colonequals\beta(K, G)$, the following asymptotic holds:
\begin{equation}\label{E:numberfieldcounting}
\N_{d,K}(G, X) \sim C X^{\alpha} (\log X)^{\beta}.
\end{equation}
These two conjectures are referred to as the weak and strong forms of Malle's conjecture, respectively.
We note that Malle's specific prediction for the value of $\beta(K,G)$
was shown to be incorrect (see~\cite{MR2135320}).
Hence, following~\cite{alberts2025inductivemethodscountingnumber},
we will refer to the assertion that the asymptotic (\ref{E:numberfieldcounting})
holds for some constants $\alpha(K,G)$, $\beta(K,G)$, $C(K,G)$
as the Number Field Counting Conjecture.

Fully resolving any of these conjectures would resolve the inverse Galois problem, so one should
expect their solution to be quite difficult.  Nevertheless, there have been major advances.
Malle's conjecture is  known to hold for $S_3$, $S_4$, and $S_5$
by the results of Davenport--Heilbronn and Bhargava already mentioned
(see~\cite{35b3f195-a21d-3d48-b94d-d206914c380e, MR2183288, MR2745272}),
and for $D_4$ by a result of Cohen, Diaz y Diaz, and Olivier (see~\cite{MR1918290}).
Moreover, Mak\"i and Wright proved the conjecture for all abelian groups
(see \cite{MR969545, MR1200974}).

More recently, building on inductive methods of Alberts and Wang (see~\cite{MR4377858, alberts2025twisted, MR4298097,MR4219215}),
Alberts, Lemke Oliver, Wang, and Wood prove an asymptotic growth rate for infinitely many groups all belonging to the family of so-called concentrated groups. Precisely which groups are covered depends on the average size of certain class group torsion and what is known towards Malle's conjecture for quotients of $G$; see~\cite{alberts2025inductivemethodscountingnumber} for a more detailed summary of these conditions.
Their theorem captures many separate results in the literature simultaneously, as the class of concentrated groups appears to be quite broad. For example, many nilpotent groups and groups assembled from wreath products are concentrated. To give some concrete examples, we have that $C_3 \wr C_2$, $D_4 \subsetneq S_4$, $S_n \times A$ for $n \leq 5$, $A$ abelian, and all abelian groups not isomorphic to a direct sum of cyclic $p$-groups are concentrated groups (see~\cite[Table 1.4]{alberts2025inductivemethodscountingnumber}). 

In cases where the Number Field Counting Conjecture is known, often the next result one seeks
is a version of the counting theorem with a power saving error term. In~\cite{MR2641942}, Belabas--Bhargava--Pomerance obtained the first power saving error term for the count of $S_3$-cubic and $S_4$-quartic fields.  The counting theorems in the cubic case  were later refined by Bhargava--Shankar--Tsimerman and Taniguchi--Thorne to include a secondary term, as conjectured by Roberts and Datskovsky--Wright (see~\cite{MR3090184, MR3127806, MR1836927, MR936994}); these results were further refined by Bhargava--Taniguchi--Thorne (see~\cite{MR4768704}).  The first power saving error term for $D_4$-quartic extensions of $\Q$ was obtained by Cohen, Diaz y Diaz, and Olivier (see~\cite{MR1918290}) in their original proof and was subsequently generalized to an arbitrary base field by Bucur, Florea, Serrano L\'opez, and Varma (see~\cite{MR4774808}). The error term of $O(X^{3/4 + \eps})$ in~\cite{MR1918290} was also recently improved by McGown and Tucker (see~\cite{MR4808549}) to $O(X^{5/8 + \eps})$.  Shankar and Tsimerman prove a power saving error term for counting $S_5$-fields (see~\cite{MR3264252}).  Power saving error terms for counting abelian extensions were established by
Alberts (see~\cite{alberts2025twisted}).

Our main goal in this paper is to study the error term in
counting octic extensions of $\mathbb{Q}$ with Galois group
$C_2 \wr S_4$.\footnote{Throughout the paper, we will always view $C_2 \wr S_4$ as a transitive subgroup of $S_8$.}
Malle's conjecture in this situation was first established by Kl\"uners
(see~\cite{MR2904935}).  Although not stated explicitly, his proof gives a power savings
of $O(X^{7/8})$.  We establish the following improvement:
\begin{theorem}\label{T:1}
The number of octic fields with Galois group $C_2\wr S_4$
and discriminant bounded by $X$
equals
$$
  \N_8(C_2 \wr S_4,X)=CX+O(X^{3/4-1/30}),
$$
where
$$
C
\colonequals
\sum_{\substack{
    [K:\Q] = 4\\
    \Gal(K/\Q) \cong S_4\\
    }}
    \frac{\zeta_K^*(1)}{2^{r_2(K)} \zeta_K(2) \Delta_K^2}.
$$
\end{theorem}
Here $\zeta^*_K(1)$ denotes the residue of the Dedekind zeta function of $K$ at $s=1$.
Perhaps part of the difficulty of establishing such a result is that there are
six subgroups of $C_2\wr S_4$ that can arise as the Galois closure 
for an extension of the form $L/K/\Q$ where $K/\Q$ is $S_4$-quartic
and $L/K$ is quadratic (see Lemma~\ref{L:1}).  This collection of groups of interest is
$$
  \mathcal{G} \colonequals \{S_4, \GL_2(\F_3), C_2\times S_4, Q_8\rtimes S_4, C_2^3\rtimes S_4, C_2\wr S_4\}
  \,.
$$
As part of the proof of Theorem~\ref{T:1}, we establish the following new upper bounds
concerning the three groups in $\mathcal{G}$ for which Malle's conjecture is not known.

\begin{theorem}\label{T:2}
The number of octic fields with Galois closure as specified satisfies:
\begin{align*}
  \N_8(S_4, X) &\ll X^{1/2}\,,\\
  \N_8(\GL_2(\F_3),X)&\ll X^{3/5}\,,\\
  \N_8(C_2^3\rtimes S_4, X)&\ll X^{9/14}\,,\\
  \N_8(Q_8\rtimes S_4,X)&\ll X^{3/5}\,.
\end{align*}
\end{theorem}

See Table~\ref{tab:my_label} and the discussion in Section~\ref{S:known}
to compare the results of Theorem~\ref{T:2} with the best known bounds.

\section{Known results for groups in $\mathcal{G}$}\label{S:known}

Let $\mathcal{F}$ denote the set of
octic extensions $L$ of the form $L/K/\Q$ where
$K/\Q$ is $S_4$-quartic and $L/K$ is quadratic.  This will be the collection
of fields of interest.  To begin, we require the following result:

\begin{lemma}\label{L:1}
Let $L$ be an octic field.
Then $L\in\mathcal{F}$  if and only if $\Gal(\widetilde{L}/\Q)\in\mathcal{G}$,
where
$$\mathcal{G}=\{S_4, \GL_2(\mathbb{F}_3), \, S_4 \times C_2, \, C_2^3 \rtimes S_4, \, Q_8 \rtimes S_4, \, C_2 \wr S_4\}\,.$$
\end{lemma}
\begin{proof}
Suppose $L\in\mathcal{F}$.
We have that $G\colonequals \Gal(\widetilde{L}/\Q)$ embeds into $C_2 \wr S_4$.
See Theorem~1.8 of~\cite{MR4858206} or Proposition~3.2 of~\cite{MR4594565}.
Using SageMath, one finds that there are $32$ subgroups of $C_2 \wr S_4$ that are transitive as subgroups
of $S_8$.  Of these $32$ subgroups, those with $S_4$ as a quotient are precisely the groups in $\mathcal{G}$.

Conversely, suppose that $G \in \mathcal{G}$. Then $H_L \colonequals \text{Gal}(\widetilde{L}/L)$ is an index 8 subgroup of $G$ containing no nontrivial normal subgroups of $G$. One can check that each such subgroup of $G$ is contained in an index 4 subgroup $H_K$ such that the quotient $G/M$ is isomorphic to $S_4$, where $M$ is the normal core of $H_K$.
Thus, the subfield of $\widetilde{L}$ fixed by $H_K$ is a quartic $S_4$-field contained in $L$, so $L\in\mathcal{F}$.
\end{proof}

We now briefly summarize what is known for the groups in $\mathcal{G}$.
Table~\ref{tab:my_label} gives each group with its permutation ID, standard group name,
the conjectured value $\alpha \colonequals \alpha(G)$, the current best known result
of the form $\N_8(G,X)\ll X^{\beta+\eps}$ by
applying general methods,
and the new bounds we establish.
We have indicated with a checkmark cases where Malle's conjecture is known.
As previously mentioned, Kl\"uners proves Malle's conjecture for $C_2 \wr S_4$.
Malle's conjecture for $S_4\times C_2$ follows from a result of Masri, Thorne, Tsai, and Wang, who
treat the case of $S_n\times A$ with $A$ an abelian group (see~\cite{MR4219215, masri2020mallesconjecturegtimes}).
Of these six groups, $Q_8 \rtimes S_4$ and $C_2 \wr S_4$ are concentrated\footnote{From~\cite{alberts2025inductivemethodscountingnumber}:
``We say that a transitive permutation group $G$ is concentrated in a proper
normal subgroup $N$ if $N$ contains all minimum index elements of $G$, and we say that $G$ is concentrated if this holds for some proper normal subgroup $N$.''} and the others are not.
However, Malle's conjecture for $Q_8\rtimes S_4$ does not follow immediately 
from the result of Alberts, Lemke Oliver, Wang, and Wood
(see~\cite{alberts2025inductivemethodscountingnumber}).

We are left with four groups, $\text{GL}_2(\mathbb{F}_3)$, $C_2^3 \rtimes S_4$, $S_4$ (in its $S_8$ embedding), and $Q_8\rtimes S_4$,
for which Malle is not known.
For comparison, we review the best bounds one can get for these groups from various methods.
Schmidt's result (see~\cite{MR1330934}) produces an upper bound, independent of $G$, of $\N_8(G,X)\ll X^{5/2}$.
Dummit (see~\cite{MR3882158}) improves this bound in some cases.
Later, a result of Alberts improves this to $X^{3/2+\eps}$ for $\text{GL}_2(\mathbb{F}_3)$
and $X^{11/8+\eps}$ for $S_4$ (see~\cite{MR4047213}).
A result of Bhargava, Shankar, and Wang
improves this to $X^{311/128 + \eps}$
for $C_2^3 \rtimes S_4$ and $Q_8 \rtimes S_4$ (see~\cite{MR4493242}).
These appear to be the best known bounds for these groups coming from general methods.
Notice that none of these produce an exponent less than one.
However, using Kl\"uners' argument (see~\cite{MR2904935}), one can achieve the upper bound
of $X^{3/4+\eps}$ for all four groups.
In all four cases, the results in Theorem~\ref{T:2} represent improvements to all known bounds.

\begin{table}[h]
    \centering
    \begin{tabular}{c|c|c|c|cc}
        Permutation ID & Group Name & Conjectured & Best known & New bound  \\
        \hline 
        8T14 & $S_4$ & $1/4$ & $11/8$ & 1/2\\
        8T23 & $\text{GL}_2(\mathbb{F}_3)$ & $1/3$ & $3/2$ & $3/5$\\
        8T24 & $S_4 \times C_2$ & $1/2$ & \checkmark & - \\
        8T39 & $C_2^3 \rtimes S_4$ & $1/2$ & $311/128$ & $9/14$\\
        8T40 & $Q_8 \rtimes S_4$ & $1/2$ & $311/128$ & $3/5$ \\
        8T44 & $C_2 \wr S_4$ & 1 & \checkmark & - \\
    \end{tabular}
    \vspace{0.5cm}
    \caption{Bounds on $\N_8(G,X)$ for $G\in\mathcal{G}$}
    \label{tab:my_label}
\end{table}

The strategy of the proof of Theorem~\ref{T:1} is as follows.  Let $\N_8(\mathcal{G},X)$ be
the number of $L\in\mathcal{F}$ with $|\Delta_L|\leq X$.  We first prove a counting theorem
with a power saving error term for $\N_8(\mathcal{G},X)$.  This is Theorem~\ref{T:3} of Section~\ref{S:maincount}.
Then, in order to obtain a power saving error term
for $\N_8(C_2 \wr S_4,X)$, we must give a strong enough upper bound on $\N_8(G,X)$
for each of the four remaining groups.
The necessary bounds, stated in Theorem~\ref{T:2}, are proved in
Sections~\ref{S:S4}, \ref{S:C23S4}, \ref{S:GL2F3}, and~\ref{S:Q8S4}.

\section{Counting octic extensions $L/K/\Q$ that are quadratic over $S_4$-quartic}\label{S:maincount}

Let $\mathcal{F}$ denote the set of
octic extensions $L$ of the form $L/K/\Q$ where
$K/\Q$ is $S_4$-quartic and $L/K$ is quadratic.
By Lemma~\ref{L:1}, the set $\mathcal{F}$ constitutes the set of octic extensions $L$
such that $\Gal(\widetilde{L}/K)\in\mathcal{G}$.
Our count in this section will include each $S_4$-quartic field only once up to isomorphism,
but in some cases we still obtain isomorphic octic fields of the form $L/K/\Q$.
This will not be an issue later as each $C_2\wr S_4$-octic field shows up exactly one time in this way,
and for the other groups $G$ we will only seek upper bounds so the multiplicity 
can be absorbed into the implicit constant.
For clarity, note that in Theorem~\ref{T:1}, we count fields up to isomorphism,
and that in the definition of the constant $C$, the sum is taken over all $S_4$-quartic
fields up to isomorphism.

The two main tools needed here will be
Bhargava's count of $S_4$-quartic fields
and a counting theorem for relative quadratic
extensions due to McGown and Tucker, where the dependence of the error term 
on the base field is made explicit.
Amusingly, no additional bound on the error term for counting $S_4$-quartic fields
is needed.  That being said, we will require a bound on
$\left|\Cl(K)[2]\right|$, the size of the $2$-torsion in the class group of an $S_4$-quartic field.
One could use Brauer--Siegel here, but we take advantage of the state-of-the-art
bound due to Bhargava--Shankar--Taniguchi--Thorne--Tsimerman--Zhao. 

\begin{theorem}\label{T:3}
The number of 
$L\in\mathcal{F}$ with $|\Delta_L|\leq X$ equals
$$
  \N_{8}(\mathcal{G}, X)=CX+O(X^{3/5 + 2\kappa/5 + \eps}),
$$
where $C$ is given in Theorem ~\ref{T:1}
and $\kappa\approx 0.2784$.
The value of $\kappa$ is given by an exact expression
in~\cite[page 8]{MR4155220}.
\end{theorem}

\begin{proof}
We begin by writing
\begin{equation} \label{eq: full sum}
    \N_{8}(\mathcal{G}, X)
    =
    \sum_{\substack{[K:\Q] = 4\\G(\widetilde{K}/\Q) \cong S_4\\|\Delta_K| \leq \sqrt{X}\\}}
    \sum_{\substack{[L:K]=2\\\Nm(\Delta_{L/K}) \leq X/\Delta_K^2\\}}1
    \,.
\end{equation}
By Theorem~2 of~\cite{MR4808549}, the inner sum equals
\begin{align}
\nonumber
\sum_{\substack{[L:K]=2\\\Nm(\Delta_{L/K}) \leq X/\Delta_K^2}}1
&=
\frac{X}{\Delta_K^2}\frac{1}{2^{r_2(K)}}\frac{\zeta_K^*(1)}{\zeta_K(2)} 
+
O \left( 
|\text{Cl}(K)[2]| \cdot |\Delta_K|^{1/5} \left(\frac{X}{\Delta_K^2}\right)^{3/5}
\log (X/\Delta_K^2)^{3}
\right)
\\
\label{E:AT1}
&=
\frac{X}{\Delta_K^2}\frac{1}{2^{r_2(K)}}\frac{\zeta_K^*(1)}{\zeta_K(2)} 
+
O \left( 
|\Delta_K|^{\kappa-1+\eps} X^{3/5+\eps}
\right)
\,.
\end{align}
Here we have used
Theorem~1.1 of~\cite{MR4155220}, which states that
\begin{equation}\label{E:torsion}
|\text{Cl}(K)[2]| \ll |\Delta_K|^{\kappa + \eps}.
\end{equation}
Moreover, using the parametrization given in~\cite{MR1918290}, or alternatively
Equation 3 in Section 3.2 of~\cite{MR4808549},
we also have the weaker estimate
\begin{align}\label{E:AT2}
\sum_{\substack{[L:K]=2\\\Nm(\Delta_{L/K}) \leq X/\Delta_K^2}}1
&\leq
|S_2(K)|\sum_{\Nm(\mathfrak{a}) \leq X/\Delta_K^2}1
&\ll
\left|\Cl(K)[2]\right|\cdot|\Delta_K|^\eps\left(\frac{X}{\Delta_K^2}\right)
&\ll
|\Delta_K|^{\kappa-2+\eps}X
\,.
\end{align}
In the above, $S_2(K)$ is the $2$-Selmer group of $K$, which can
be bounded using Proposition 3.2 of~\cite{MR1918290}.  Moreover,
we have used $\sum_{\Nm(\mathfrak{a}) \leq Y}1\leq|\Delta_K|^\eps Y$;
this is not hard to derive and appears as Equation~(21) in \cite{MR4808549}.

Let $0<Z<X^{1/2}$ be a parameter.  Splitting the sum in (\ref{eq: full sum}), we have
\begin{align}\label{E:split}
    \N_{8}(\mathcal{G}, X)
    &=
    \sum_{
        \substack{
            [K:\Q] = 4\\
            G(\widetilde{K}/\Q) \cong S_4\\
            |\Delta_K| \leq Z\\
        }
    }
    \sum_{
    \substack{
    [L:K]=2\\
    \Nm(\Delta_{L/K}) \leq X/\Delta_K^2\\
    }
    }{1}
    +
        \sum_{
        \substack{
            [K:\Q] = 4\\
            G(\widetilde{K}/\Q) \cong S_4\\
            Z<|\Delta_K| \leq \sqrt{X}\\
        }
    }
    \sum_{
    \substack{
    [L:K]=2\\
    \Nm(\Delta_{L/K}) \leq X/\Delta_K^2\\
    }
    }{1}
    \,.
\end{align}
Using (\ref{E:AT1}), the first piece of (\ref{E:split}) equals
\begin{align}
\nonumber
    X
   \sum_{
        \substack{
            [K:\Q] = 4\\
            G(\widetilde{K}/\Q) \cong S_4\\
            |\Delta_K| \leq Z\\
        }
    }
    \frac{1}{2^{r_2(K)}\Delta_K^2}\cdot\frac{\zeta_K^*(1)}{\zeta_K(2)} 
    +
        O
    \left(
    X^{3/5+\eps}
    \sum_{
        \substack{
            [K:\Q] = 4\\
            G(\widetilde{K}/\Q) \cong S_4\\
            |\Delta_K| \leq Z\\
        }
    }
    |\Delta_K|^{\kappa-1+\eps}
    \right)
    \,.
\end{align}

From~\cite{MR2183288} one knows $\N_4(S_4,X)\ll X$, and therefore partial summation yields
\begin{align}
\sum_{\substack{[K:\Q] = 4\\G(\widetilde{K}/\Q) \cong S_4\\|\Delta_K| > Z\\ }}
\frac{1}{2^{r_2(K)}\Delta_K^2}\cdot\frac{\zeta_K^*(1)}{\zeta_K(2)} 
\ll
\frac{1}{Z^{1-\eps}}
\,,\qquad
\sum_{\substack{[K:\Q] = 4\\G(\widetilde{K}/\Q) \cong S_4\\|\Delta_K| \leq Z\\}}|\Delta_K|^{\kappa-1+\eps}
\ll
Z^{\kappa+\eps}\,.
\end{align}
Note we have also used the estimate $\zeta^*_K(1)\ll |\Delta_K|^\eps$.
Therefore we have
$$
    \sum_{
        \substack{
            [K:\Q] = 4\\
            G(\widetilde{K}/\Q) \cong S_4\\
            |\Delta_K| \leq Z\\
        }
    }
    \sum_{
    \substack{
    [L:K]=2\\
    \Nm(\Delta_{L/K}) \leq X/\Delta_K^2\\
    }
    }{1}
    =CX+O\left(XZ^{-1+\eps}\right)+O\left(X^{3/5+\eps}Z^{\kappa+\eps}\right)
    \,.
$$
For the second piece of (\ref{E:split}), we use (\ref{E:AT2}) to obtain
$$
        \sum_{
        \substack{
            [K:\Q] = 4\\
            G(\widetilde{K}/\Q) \cong S_4\\
            Z<|\Delta_K| \leq \sqrt{X}\\
        }
    }
    \sum_{
    \substack{
    [L:K]=2\\
    \Nm(\Delta_{L/K}) \leq X/\Delta_K^2\\
    }
    }{1}
  \ll
  X
          \sum_{
        \substack{
            [K:\Q] = 4\\
            G(\widetilde{K}/\Q) \cong S_4\\
            Z<|\Delta_K| \leq \sqrt{X}\\
        }
    }
    |\Delta_K|^{\kappa-2+\eps}
\ll
XZ^{\kappa-1+\eps}
\,.
$$
Summarizing, the error term is
$$
O\left(XZ^{-1+\eps}\right)+O\left(X^{3/5+\eps}Z^{\kappa+\eps}\right)
+O\left(
XZ^{\kappa-1+\eps}
\right)
\,.
$$
Choosing $Z=X^{2/5}$, we have
\begin{equation}\label{E:big.count}
\N_8(\mathcal{G}, X) = C X + O(X^{3/5 + 2\kappa/5 + \eps}),
\end{equation}
where $C$ is as in Theorem~\ref{T:1}.
\end{proof}

\section{An upper bound for octic fields with Galois group $S_4$}\label{S:S4}

In this section, we prove an upper bound on $\N_8(S_4,X)$, which is the first
of four groups we have to treat in Theorem~\ref{T:2}.
The next proof is short and far from an optimal bound, but we require it nonetheless 
for establishing Theorem~\ref{T:1}.

\begin{proof}[Proof of Theorem~\ref{T:2} for $G=S_4$]
Given an $S_4$-octic field $L$, we know that $L$ contains a unique $S_4$-quartic field $K$.
Moreover, $L$ is the unique $S_4$-octic field satisfying $K\subseteq L\subseteq\widetilde{K}$.

Therefore, given an $S_4$-quartic field $K$, we can define
$$
  f_K(X) \colonequals \begin{cases}
  1, & \Nm(\Delta_{L/K})\le X/\Delta_K^2,\\
  0 & \text{otherwise},
  \end{cases}
$$
where $L$ is the unique $S_4$-octic containing $K$.
We then have
\begin{align*}
    \N_8(S_4,X) = \sum_{\substack{[K:\Q]=4 \\ G(\widetilde{K}/\Q) \cong S_4 \\ |\Delta_K|\le \sqrt{X}}}f_K(X).
\end{align*}

Trivially bounding this sum and using $\N_4(S_4,X)\ll X$ (see~\cite{MR2183288}) yields 
\begin{align*}
    \N_8(S_4,X) \ll X^{1/2},
\end{align*}
as desired.
\end{proof}

\section{An upper bound for octic fields with Galois group $C_2^3\rtimes S_4$}\label{S:C23S4}

In this section we treat the second of the four groups $G$ appearing in Theorem~\ref{T:2};
namely, $G=C_2^3\rtimes S_4$.
The following observation will be helpful in the argument that follows.

\begin{lemma}\label{L:semidirect}
The discriminant of an octic field $L$ with Galois closure $C_2^3 \rtimes S_4$ is always a square. 
\end{lemma}

\begin{proof}
Let $L$ be such an extension.
One verifies that any embedding of $C_2^3 \rtimes S_4$ in $S_8$ is contained in $A_8$.
If $f(x) \in \mathbb{Q}[x]$ is the minimal polynomial of a primitive element $\alpha$ of $L/\Q$,
then it is well-known that the previous statement implies that $\disc(f)$ is a square.
But $\disc(f)$ is also the discriminant of the lattice $\mathbb{Z}[\alpha]$, which differs from $\Delta_L$ by a square. Hence $\Delta_L$ is also a square.
\end{proof}

\begin{proof}[Proof of Theorem~\ref{T:2} for $G=C_2^3\rtimes S_4$]
Using Lemma~\ref{L:semidirect} we have
$$
  \N_8(C_2^3 \rtimes S_4,X)\ll
\sum_{\substack{[K:\Q]=4\\ G(\widetilde{K}/\Q)\cong S_4\\|\Delta_K|\leq\sqrt{X}}}
\sum_{\substack{[L:K]=2\\
\Nm(\Delta_{L/K})\leq X/\Delta_K^2\\
\Nm(\Delta_{L/K})=\square}}
1.
$$
Using the parametrization given in~\cite{MR1918290}, or alternatively
Equation 3 in Section 3.2 of~\cite{MR4808549},
we can bound the inner sum as
\begin{align*}
\sum_{\substack{[L:K]=2\\
\Nm(\Delta_{L/K})\leq X/\Delta_K^2\\
\Nm(\Delta_{L/K})=\square}}1
&\ll |S_2(K)|\sum_{\substack{\Nm(\mathfrak{a})\leq X/\Delta_K^2\\
\Nm(\mathfrak{a})=\square}}1
\,,
\end{align*}
where $S_2(K)$ denotes the $2$-Selmer group of $K$ and the notation $\Nm(\mathfrak{b})=\square$
denotes that the absolute norm of the ideal $\mathfrak{b}$ in $K$ is a square in $\Z$.
Using Proposition 3.2 of~\cite{MR1918290} and
(\ref{E:torsion}),
we find that
$$
   |S_2(K)|\ll |\Cl(K)[2]| \ll |\Delta_K|^{\kappa+\eps}. 
$$
 Moreover, we have
$$
\sum_{\substack{\Nm(\mathfrak{a})\leq Y\\
\Nm(\mathfrak{a})=\square}}1
=
\sum_{\substack{n\leq Y\\n=\square}}a_n(K)
\leq 
\sum_{\substack{n\leq \sqrt{Y}}}\tau_4(n^2)
\ll
Y^{1/2+\eps}.
$$
In the above, the $a_n(K)$ are the coefficients in the Dirichlet
series of the Dedekind zeta function $\zeta_K(s)$ and
$\tau_4(n)$ is the number of ways to write $n$ as a product of $4$ positive integers.
We have used the estimate $a_n\leq \tau_4(n)\ll n^\eps$; see, for example, Exercise 1 on page 231 of \cite{MR195803}, and the result contained in~\cite{MR4188678}.
Assembling all the pieces together and using partial summation on $\N_4(S_4,X)\ll X$,
we conclude that
\begin{equation}
\N_8(C_2^3 \rtimes S_4,X)\ll
X^{1/2+\eps}
\sum_{\substack{[K:\Q]=4\\ G(\widetilde{K}/\Q)\cong S_4\\|\Delta_K|\leq\sqrt{X}}}
|\Delta_K|^{\kappa-1-\eps}
\ll
X^{1/2 + \kappa/2 + \eps}
\,.
\end{equation}
\end{proof}

\begin{remark}
We believe our upper bound in this case could be improved,
but
have chosen not to pursue it here.  However, we note that
if one assumes that the average of $\left|\Cl(K)[2]\right|$ over $S_4$-quartic fields $K$ is $\ll|\Delta_K|^\eps$,
a slight modification of the above argument would yield an upper bound of $\N_8(G,X)\ll X^{1/2+\eps}$,
the predicted upper bound in the weak form of Malle's conjecture
for $G=C_2^3\rtimes S_4 $.
\end{remark}

\section{An upper bound for octic fields with Galois group $\GL_2(\F_3)$}\label{S:GL2F3}

Here we treat the third group in the statement of Theorem~\ref{T:2}.
The conjectured Malle constant $\alpha$ for $\GL_2(\F_3)$ is $1/3$,
which is smaller than that of all other groups in $\mathcal{G}$, save $S_4$.
As such, we are interested in deriving a bound for this group that would beat the bound we found
for $C_2^3\rtimes S_4$ in Section~\ref{S:C23S4}.

Throughout this section, unless otherwise specified,
$G$ will always denote $\GL_2(\F_3)$
and $\widetilde{L}/L/K/\Q$ will always denote a tower of fields
where $K/\Q$ is $S_4$-quartic, $L/K$ is quadratic, and
$\Gal(\widetilde{L}/\Q)\cong G$.
We now prove a few simple facts regarding the splitting of ramified primes in an octic $G$-extension. In the subsequent proofs,
let $\widetilde{\mathfrak{P}}\mid\mathfrak{P}\mid\mathfrak{p}\mid p$ be primes
in $\widetilde{L}/L/K/\Q$ respectively, each one lying over (i.e., dividing) the next.
\begin{lemma} \label{lemma: splitting in G}
  Let $L/K/\Q$ be as above. Let $p$ be a prime that is tamely ramified in $\widetilde{L}$.
    \begin{enumerate}
        \item We have $e(\widetilde{L}/ p) \in \{2,3,4,6,8\}$.
        \item Suppose $e(\widetilde{L}/p) = 2$.  If the ramification indices $e(\mathfrak{p}/p)$
        for $\mathfrak{p}$ in $K$ over $p$ are not all the same,
        then we must have $f(\widetilde{L}/p) \leq 2$.
        \item
        Suppose $e(\mathfrak{P}/p)\leq 2$ for all primes $\mathfrak{P}$ in $L$ above $p$.
        If the indices $e(\mathfrak{p}/p)$ for $\mathfrak{p}$ in $K$ over $p$ are not all equal,
        then we must have $e(\widetilde{L}/p) = 2$ and $f(\widetilde{L}/p) \leq 2$.
    \end{enumerate}
\end{lemma}
\begin{proof}
    (1) The set $\{1,2,3,4,6,8\}$ is precisely the set of orders of the cyclic subgroups of~$G$.
    
    (2) Suppose that $e(\widetilde{L}/p) = 2$. Then the inertia groups $I(\widetilde{\mathfrak{P}}/p)$ are all isomorphic to $C_2$. Since the indices $e(\mathfrak{p}/p)$ are not all the same, if all the inertia groups were the same, we would reach a contradiction by restricting to $\widetilde{L}/K$. Hence, the inertia groups $I(\widetilde{\mathfrak{P}}/p)$ form a size 12 conjugacy class. Thus, the normalizer $N$ of each inertia group $I(\widetilde{\mathfrak{P}}/p)$ has index 12 in $G$. Each corresponding decomposition group
    $D(\widetilde{\mathfrak{P}}/p)$
    is contained in $N$, so $|D(\widetilde{\mathfrak{P}}/p)| \leq 4$. We have thus proved $f(\widetilde{L}/p) \in \{1,2\}$.

    (3) The ramification index $e(\widetilde{L}/p)$ cannot be 8, as then $4$ would have to divide the degree of the extension $\widetilde{L}/L$. If $e(\widetilde{L}/p) = 6$, then all the inertia groups $I(\widetilde{\mathfrak{P}}/p)$ must have intersection isomorphic to either $C_3$ or $C_6$ with $\text{Gal}(\widetilde{L}/L)$. 
    However, we have that $\text{Gal}(\widetilde{L}/L) \cong S_3$ in $G$ and every copy of $S_3$ only intersects one copy of $C_6$ nontrivially;
    this copy cannot constitute a conjugacy class as no $C_6$ is normal in $G$. Finally, one can verify that every intersection of a copy of $C_4$ with a copy of $S_3$ is trivial in $G$, which excludes the case $e(\widetilde{L}/p) = 4$ since ramification happens in $\widetilde{L}/L$.
\end{proof}

\begin{lemma} \label{lemma: vp(N)}
  Let $L/K/\Q$ as in Lemma~\ref{lemma: splitting in G}.
  Write $\Nm(\Delta_{L/K})=N_0 N_1 N_2$ with
  $$(N_0,6)=1,\;\; \rad(N_0)\mid\Delta_K,\;\; (N_1,6\Delta_K)=1,\;\;N_2\mid 6^m\,.$$
  Then we have
  \begin{enumerate}
  \item
  $v_p(N_1)=4$ for all $p\mid N_1$,
  \item
  $v_p(N_0)\leq v_p(\Delta_K)$ for all $p$,
  \item
  If $v_p(\Delta_K) = 3$ and $p\nmid N_2$, then $p \mid N_0$.
  \end{enumerate}
\end{lemma}
\begin{proof}
(1) Let $p$ be a prime dividing $N_1$.
    By definition, $p$ ramifies in $L$ but is unramified in $K$. Then $p$ must also be unramified in $\widetilde{K}$, a degree 24 subfield of $\widetilde{L}$, and so we deduce that $e(\widetilde{L}/p) = 2$. The restriction of any of the inertia groups $I(\widetilde{\mathfrak{P}}/p)$ to the field extension $\widetilde{L}/\widetilde{K}$ must have order 2 since $e(\widetilde{\mathfrak{P}}/p) = 2$ and
    $p$ is unramified in $\widetilde{K}$.
    We conclude that the inertia groups $I(\widetilde{\mathfrak{P}}/p)$ are all equal to $\text{Gal}(\widetilde{L}/\widetilde{K})$ since it is the unique normal subgroup of order 2 in $\text{GL}_2(\mathbb{F}_3)$. Hence
    $e(\widetilde{\mathfrak{P}}_1/\mathfrak{P}_1) = e(\widetilde{\mathfrak{P}}_2/\mathfrak{P}_2)$ for any
    primes $\widetilde{\mathfrak{P}}_1, \widetilde{\mathfrak{P}}_2$ of $\widetilde{L}$
    (where $\mathfrak{P}_1,\mathfrak{P}_2$ are the primes of $L$ lying underneath).
    It follows that the ramification indices $e(\mathfrak{P}/p)$ are the same
    for all primes $\mathfrak{P}$ of $L$ over $p$.
    Thus, since $p$ is unramified in $K$, we find that all primes of $K$ above $p$ ramify in $L/K$.

(2) Let $p$ be a prime that divides $N_0$.
Since $v_p(\Delta_L)\leq 7$, from the relation
\begin{equation}\label{E:rel}
2 v_p(\Delta_K) + v_p(N_0) = v_p(\Delta_L)\,,
\end{equation}
we know $v_p(\Delta_K)\leq 3$.
We treat cases based on the value of $v_p(\Delta_K)$,
which must be positive since $p$ ramifies in $K$.
In the case where $v_p(\Delta_K)=3$, we deduce $v_p(N_0) = 1$
and the conclusion follows immediately.
The cases where $v_p(\Delta_K)=1,2$ remain.

Factor $(p) = \mathfrak{p}_1^{e_1} \mathfrak{p}_2^{e_2} \cdots \mathfrak{p}_m^{e_m}$ in $K$
and write $f_i \colonequals f(\mathfrak{p}_i/p)$.
We will denote the splitting type of $p$ in $K$ by $\sigma_K(p) \colonequals (f_1^{e_1} \dots f_m^{e_m})$.
Let $\delta_i$ denote the indicator function for
whether the prime $\mathfrak{p}_i$ ramifies in $L/K$.  We have that 
\begin{equation}\label{E:cond}
  v_p(N_0) = \sum_i{\delta_i f_i}
  \,,
  \qquad
  v_p(\Delta_K) = \sum_{i}{(e_i-1)f_i}
  \,.
\end{equation}  
  We would like to show that the first sum is at most the value of the second sum.
  We note that it suffices to show that $\delta_i = 1$ implies $e_i>1$; that is, for any prime $\mathfrak{p}$ in $K$ above $p$, if $\mathfrak{p}$ ramifies in $L/K$, then $\mathfrak{p}$ ramifies in $K/\Q$.
    
    Now suppose that $v_p(\Delta_K) = 2$.
    Then we must have
    $\sigma_K(p)\in\{(1^2 1^2), (1^3 1), (2^2)\}$.  In all cases, we find $v_p(N_0)\leq v_p(\Delta_K)$
    from (\ref{E:cond}).

    For the final case, we suppose that $v_p(\Delta_K) = 1$.
    Since $p$ is tamely ramified in $L/\mathbb{Q}$,
    the inertia group of $p$ is generated by a permutation of index $v_p(\Delta_L)$.
    As the indices of all elements in $\text{GL}_2(\mathbb{F}_3)$ are contained in $\{3,4,6,7\}$, we have via (\ref{E:rel}) that $v_p(N_0) \in \{1,2,4,5\}$.
    On the other hand, since $p$ is ramified in $K$, it follows from (\ref{E:cond}) that $v_p(N_0)\leq 3$.
    Hence it suffices to show that $v_p(N_0)\neq 2$.
    
    
    Suppose towards a contradiction that $v_p(N_0) = 2$.
    As $v_p(\Delta_K) = 1$ we must have $\sigma_K(p)\in\{(1^22), (1^211)\}$.
    In the case that $\sigma_K(p) = (1^22)$, we may assume that the degree $2$ prime $\mathfrak{p}_2$ in $K$ above $p$
    ramifies in $L$, as otherwise the conclusion holds by (\ref{E:cond}).
    We also observe that only this prime ramifies in $L/K$, as otherwise $v_p(N_0) > 2$. Thus $\sigma_L(p) \in \{(1^21^2 2^2), (2^2 2^2)\}$. By part (3) of Lemma ~\ref{lemma: splitting in G}, we know that $e(\widetilde{L}/p) = f(\widetilde{L}/p) = 2$.  Since the inertia groups $I(\widetilde{\mathfrak{P}}/p) \cong C_2$ are not all the same, they constitute a conjugacy class of size 12. Since all the ramification indices of primes in $L$ over $p$ are the same, the size of all the intersections $I(\widetilde{\mathfrak{P}}/p) \cap S_3$ must be the same for some copy of $S_3 \leq G$. However one checks that the 12 non-normal copies of $C_2$ in $G$ intersect any copy of $S_3$ in $G$ with orders both 1 and 2, a contradiction.
    
    Moving on, we now assume $\sigma_K(p) = (1^211)$.  As we are in the situation where two primes above $p$
    ramify in $L/K$, we have
    $\sigma_L(p)\in\{(1^4 1^2 2), (1^4 1^2 11), (1^21^21^21^2), (2^21^21^2)\}$.
    The first two may be resolved easily by observing that $4$ does not divide the degree of the extension $\widetilde{L}/L$. The last may be resolved with the same logic as for the case of $\sigma_K(p) = (1^2 2)$. Thus we assume that $\sigma_L(p) = (1^21^21^21^2)$. If $f(\widetilde{L}/p) = 2$, then the same logic applies, so we can further assume that $f(\widetilde{L}/p) = 1$ by part (3) of Lemma~\ref{lemma: splitting in G}. We observe that every non-normal copy of $C_2$ in $G$ must intersect some copy of $S_3$ trivially, since the ramification indices $e(\mathfrak{P}/p)$ are all equal to 2. As before, this does not happen and so we have achieved a contradiction.

    (3) Suppose that $v_p(\Delta_K) = 3$. Then the relation $\sum_{i} (e_i-1) f_i = 3$ holds simultaneously with $\sum_{i} e_i f_i = 4$ in $K$. We deduce that $p$ totally ramifies in $K$, so we can write $(p) = \mathfrak{p}^4$ with $\mathfrak{p} \subseteq \mathcal{O}_K$. Since $p$ is tamely ramified, the inertia groups $I(\widetilde{\mathfrak{P}}/p)$ are cyclic of order a multiple of 4. The only such subgroups of $G$ are $C_4, C_8$. We know that every $C_4, C_8$ in $G$ intersects any copy of $D_6 \cong \text{Gal}(\widetilde{L}/K)$ as $C_2$, so we deduce $e(\widetilde{L}/\mathfrak{p}) = 2$. Then $e(\widetilde{L}/p) = 8$ and $I(\widetilde{\mathfrak{P}}/p) \cong C_8$ for every prime $\widetilde{\mathfrak{P}}$ in $L$ lying above $p$. We also know that $C_8 \cap S_3 = \{1\}$ for every copy of $C_8, S_3$ in $G$ (note that $S_3 \cong \text{Gal}(\widetilde{L}/L)$). This guarantees that for any prime $\mathfrak{P}$ in $L$ above $\mathfrak{p}$, we have $e(\widetilde{L}/\mathfrak{P}) = 1$. It follows that $\mathfrak{p}$ ramifies in $L/K$, and so $p \mid \Nm(\Delta_{L/K})$.
\end{proof}

\begin{lemma}\label{L:fourth}
Write $\Nm(\Delta_{L/K})=N_0 N_1 N_2$ as in Lemma~\ref{lemma: vp(N)}.
The number of quadratic extensions $L/K$ of a fixed $S_4$-quartic field $K$ 
(satisfying $\Gal(\widetilde{L}/\Q)\cong \GL_2(\F_3)$)
with $N_1\leq Y$ satisfies
$$
  \sum_{\substack{[L:K]=2\\\Gal(\widetilde{L}/\Q)\cong G\\N_1\leq Y}}1
  \ll
  |\Delta_K|^{\kappa+\eps}Y^{1/4}.
$$
\end{lemma}

\begin{proof}
The type of estimate we make here essentially appears in~\cite{MR2904935}.
In addition to the primes dividing $N_1$, such a field $L$ is potentially
ramified at the primes dividing
$6\Delta_K$, but no others.  Therefore, using the same idea as (\ref{E:AT2}),
the number of such extensions is bounded by a constant times $\left|\Cl(K)[2]\right|$
times the number of squarefree ideals $\mathfrak{a}$ in $K$ supported at primes $\mathfrak{p}$
dividing $N_0 N_1 N_2$.  The number of squarefree $\mathfrak{a}$ supported at primes
dividing $N_0 N_2$ is bounded by $2^{4\omega(6\Delta_K)}$.
By Lemma~\ref{lemma: vp(N)}, 
we must have that $N_1$ is a $4$-th power, and thus the number of
$N_1\leq Y$ is bounded by $\sum_{m\leq Y^{1/4}}a_{m^4}(K)$.
Here the values $a_n(K)\leq \tau(n)^4$ are the coefficients of the Dedekind zeta function
of $K$, which represent the number of ideals in $K$ with norm $n$.
Therefore
$$
\sum_{\substack{[L:K]=2\\\Gal(\widetilde{L}/\Q)\cong G\\N_1\leq Y}}1
  \ll 2^{4\omega(6\Delta_K)}|\Cl(K)[2]|\sum_{m\leq Y^{1/4}}a_{m^4}(K)
  \ll |\Cl(K)[2]|\cdot |\Delta_K|^\eps Y^{1/4}
  \,.
$$
\end{proof}

In~\cite{BCT}, Bhargava, Cojocaru, and Thorne prove a tail estimate
for counting $S_5$-fields.  We require the analogous result for 
counting $S_4$-fields. 

\begin{lemma}\label{L:tail}
Let $\N^Z_4(S_4,X)$ be the number of $S_4$-quartic fields with $|\Delta_K|\leq X$
such that there exists a squarefree $q>Z$ with $q^2\mid \Delta_K$. Then
$$
  {\N}^Z_4(S_4,X)\ll X^{11/12+\eps}+\frac{X}{Z^{1-\eps}}
  \,.
$$
\end{lemma}
\begin{proof}
In~\cite{BCT}, the analogous result for $S_5$-quintic fields is proved.  We mimic the same proof,
but use the parametrization for quartic fields given in~\cite{MR2183288}.

For a $G_{\mathbb{Z}}$-invariant set $S \subseteq V_{\mathbb{Z}}$, let $\N(S,X)$ denote the number of absolutely irreducible $G_{\mathbb{Z}}$-orbits of $S$ having discriminant bounded by $X$. For a suitable box $H$ (which we determine later), with characteristic function $\Phi = \Phi_H$, we can sum through the three nondegenerate orbits of $G_{\mathbb{R}}$ on $V_{\mathbb{R}}$ to obtain
\begin{equation}\label{E:S4-tail-def}
\N(S, X) = 
\sum_{i=0}^2
\frac{
\int_{v \in V_{\mathbb{R}}^{(i)}}
{\Phi(v)\#\{x \in \mathcal{F}v \cap S \text{ abs.~irr.} :   0 < |\disc(x)| < X\} |\disc(v)|^{-1} dv}
}{
M_i
},
\end{equation}
where $M_i = M_i(H)$ is a constant depending only on $i$ and $H$. Even if $S \subseteq V_{\mathbb{Z}}$ is not invariant under the action of $G_{\mathbb{Z}}$, we define $\N(S,X)$ to be the quantity in Equation~\eqref{E:S4-tail-def}.
Following Lemma 7 of \cite{BCT}, we fix any of the 12 coordinates of $V_{\mathbb{Z}}$ as $y$ and let $\mathcal{Y} \subseteq \mathbb{A}_{\mathbb{Z}}^{12}$ be defined by
\[
\mathcal{Y} \colonequals \left\{v \colon \disc(v) = \frac{\partial \disc}{\partial y}(v) = 0 \right\}\,,
\]
which is a subscheme of $\mathbb{A}_{\mathbb{Z}}^{12}$ of codimension 2. 

We know by the discussion in Section~4.2 of \cite{bhargava2014geometricsievedensitysquarefree} that for $v \in V_{\mathbb{Z}}$ absolutely irreducible, if $p^2 \mid \disc(v)$ then $\disc$ is strongly a multiple of $p^2$ at $v$. Moreover, by Lemma 3.6 of \cite{bhargava2014geometricsievedensitysquarefree}, we know that this implies $v \pmod{p} \in \mathcal{Y}(\mathbb{F}_p)$. Thus, it suffices for our purposes to bound the quantity $\N(S_Z(\mathcal{Y}), X)$, where we define $S_Z(\mathcal{Y})$ to be the intersection of the set
\[
\{v \in G_{\mathbb{Z}} \backslash V_{\mathbb{Z}} \colon \disc(v) \neq 0 \text{ and } v \pmod{q} \in \mathcal{Y}(\mathbb{Z}/q \mathbb{Z}) \text{ for some squarefree $q > Z$}\}
\]
with the set of all absolutely irreducible elements of $ V_{\mathbb{Z}}$.
By Lemma 11 of \cite{MR2183288}, the contribution from $v = (A,B)$ with $a_{11} = 0$ is $O(X^{11/12})$, so
\[
\N(S_Z(\mathcal{Y}),X) 
\ll
\N(S_Z(\mathcal{Y}) \cap \{a_{11} \neq 0\}, X) + O(X^{11/12}).
\]
Then we have by Equation 7 of \cite{MR2183288} that
\begin{equation}\label{S4-tail-eq1}
\begin{split}
\N(S_Z(\mathcal{Y}) \cap \{a_{11} \neq 0\}, X) \\
    \ll
\int_{g \in K N'A'^{-1} \Lambda} &
\sum_{\substack{
        x \in S_Z(\mathcal{Y})\\
        a_{11} \neq 0\\
        |\disc(x)| < X}}
\Phi(kna^{-1} \lambda x) s_1^{-2}s_2^{-6}s_3^{-6}\, 
d^{\times}\lambda\, d^{\times}s\, dn\, dk.
\end{split}
\end{equation}
Now let
\[
H \colonequals \{(A,B) \in V_{\mathbb{R}} \colon |a_{ij}|,|b_{ij}| \leq 10 \text{ and $\disc$}(A,B) \geq 1\}.
\]
Then the compact set
\[
H' \colonequals \{(A,B) \in V_{\mathbb{R}} \colon |a_{ij}|,|b_{ij}| \leq 60 \text{ and $\disc$}(A,B) \geq 1\}
\]
contains $KN'H$. Equation~\eqref{S4-tail-eq1} then becomes
\[
\ll \int_{\lambda = X^{-1/12}}^{c} \int_{s_1,s_2,s_3 = 1/2}^{\infty} 
\sum_{
    \substack{
    (A,B) \in S_Z(\mathcal{Y})\\
    a_{11} \neq 0\\
    |\disc(A,B)| < X
    }
} \Phi'(\lambda a(s)^{-1}(A,B))
s_1^{-2} s_2^{-6} s_3^{-6} \ d^{\times }s \ d^{\times} \lambda,
\]
where $\Phi'$ denotes the characteristic function for $H'$. We recall that $\lambda  a(s)^{-1}(A,B) \in H'$ only if $(A,B) \in B_0$, where $B_0$ is the box defined in Equation 10 of \cite{MR2183288}. Thus, we have 
\begin{align*}
    \N(S_Z(\mathcal{Y})\cap \{a_{11}\neq 0\}, X)
& \\ \ll \int_{\lambda = X^{-1/12}}^{c} &\int_{s_1,s_2,s_3 = 1/2}^{\infty} 
 \sum_{
    \substack{
    (A,B) \in S_Z(\mathcal{Y})\\
    a_{11} \neq 0\\
    |\disc(A,B)| < X
    }
} \Phi_0(A,B)
s_1^{-2} s_2^{-6} s_3^{-6} \ d^{\times }s \ d^{\times} \lambda,
\end{align*}
\[ 
\]
where $\Phi_0$ is the characteristic function for $B_0$. Each side length of $B_0$ can be written in the form $r_i' \cdot \frac{120}{\lambda}$, and so we can invoke Theorem 13 of~\cite{BCT} to bound the integrand with $n = 12$, $k = 2$, $r_i = r_i'/\lambda$, and $B$ a centrally symmetric hypercube of side length $120$. We note that the product of the $r_i'$ is equal to 1, so $T = \lambda^{-12}$, which gives  
\[
\sum_{
    \substack{
    (A,B) \in S_Z(\mathcal{Y})\\
    a_{11} \neq 0\\
    |\disc(A,B)| < X
    }
} \Phi_0(A,B)
\ll
\lambda^{-12 - \epsilon}
\left(
\frac{1}{r_1} + \frac{1}{Z}
\right)
\ll 
\lambda^{-12 - \epsilon}
\left(
\lambda s_1 s_2^4 s_3^2 + \frac{1}{Z}
\right).
\]
Here we have used $r_i' \gg s_1^{-1}s_2^{-4} s_3^{-2}$, which follows from the inequalities defining $B_0$ and the fact that the $s_i$ are absolutely bounded from below. Plugging this back into the integral gives
\[
\N(S_Z(\mathcal{Y}) \cap \{a_{11} \neq 0\}, X) \ll 
X^{11/12 + \epsilon}
+
\frac{X}{Z^{1-\epsilon}}.
\]
Thus $\N(S_Z(\mathcal{Y}), X) \ll X^{11/12 + \epsilon}
+
X/Z^{1-\eps},
$
as desired.

\end{proof}

\begin{proof}[Proof of Theorem~\ref{T:2} for $G=\GL_2(\F_3)$]
As in Lemma~\ref{lemma: vp(N)} we write $\Nm(\Delta_{L/K})=N_0 N_1 N_2$
with $N_1$ being a $4$-th power in $\Z$.
Moreover we write $\Delta_K=3^a N_0 N_0'$ with $(N_0',3)=1$;
if $p\mid N_0'$, then  $p^2\mid\Delta_K$ by the proof of Lemma~\ref{lemma: vp(N)}. 
As we only seek an upper bound,
we will use the fact that $\Nm(\Delta_{L/K})\leq X$ implies
$N_0 N_1\leq X$ to avoid the wildly ramified primes in our count.
Let $Y>0$ be a parameter to be determined later.
We will condition on the value of $\rad(N_0')$.  

First, we treat the case where $\rad(N_0')\leq Y$.
Observe that $\Delta_K^2 N_0 N_1\leq |\Delta_L|\leq X$
and $N_0=\Delta_K/(3^a N_0')$
implies $N_1\leq 3^a N_0' X/\Delta_K^3$.
Moreover, in this case 
$N_0'\leq Y^2$
(by Lemma~\ref{lemma: vp(N)}, Part 3)
and hence $N_1\leq 3^a XY^2/\Delta_K^3$.
Making use of Lemma~\ref{L:fourth}, we estimate:
\begin{align*}
\sum_{|\Delta_K|\leq \sqrt{X}}
\sum_{\substack{N_0 N_1\leq X/\Delta_K^2\\ \rad(N_0')\leq Y}}
1
&\leq
\sum_{|\Delta_K|\leq \sqrt{X}}
\sum_{\substack{N_1\leq\frac{3^a XY^2}{\Delta_K^3}}}
1
\\
&\ll
\sum_{|\Delta_K|\leq \sqrt{X}}
|\Delta_K|^{\kappa+\eps}\left(\frac{XY^2}{|\Delta_K|^3}\right)^{1/4}
\\[1ex]
&=
X^{1/4}Y^{1/2}\sum_{|\Delta_K|\leq \sqrt{X}}
|\Delta_K|^{\kappa-3/4+\eps}
\\[1ex]
&\ll
X^{3/8+\kappa/2+\eps}Y^{1/2}
\,.
\end{align*}

Now we treat the cases $\rad(N_0')>Y$.  In this case there exists a squarefree $q>Y$
such that $q^2\mid\Delta_K$; namely, $q=\rad(N_0')$.
This makes Lemma~\ref{L:tail} applicable,
which takes the form $O(X^{1-\delta+\eps}+X/Y^{1-\eps})$
with $\delta=1/12$.
We estimate:
\begin{align*}
\sum_{|\Delta_K|\leq \sqrt{X}}
\sum_{\substack{N_0 N_1\leq X/\Delta_K^2\\ \rad(N_0')>Y}}
1
&\leq
\sum_{\substack{|\Delta_K|\leq \sqrt{X}\\\exists q>Y\,q^2\mid\Delta_K}}
\sum_{\substack{N_1\leq X /\Delta_K^2\\ }}
1
\\
&\ll
\sum_{\substack{|\Delta_K|\leq \sqrt{X} \\\exists q>Y\,q^2\mid\Delta_K}}
|\Delta_K|^{\kappa+\eps}\left(\frac{X}{\Delta_K^2}\right)^{1/4}
\\[1ex]
&=
X^{1/4}
\sum_{\substack{|\Delta_K|\leq \sqrt{X}\\\exists q>Y\,q^2\mid\Delta_K}}
|\Delta_K|^{\kappa-1/2+\eps}
\\[1ex]
&\ll
X^{1/2+\kappa/2-\delta/2+\eps}+X^{1/2+\kappa/2 + \eps}Y^{-1}
\,.
\end{align*}

Summarizing, we have a bound of 
\begin{align*}
&\ll 
X^{3/8+\kappa/2}Y^{1/2}
+
X^{1/2+\kappa/2-\delta/2+\eps}+X^{1/2+\kappa/2 + \epsilon}Y^{-1}
\\
&\ll
X^{5/12+\kappa/2+\eps}+X^{1/2+\kappa/2-\delta/2+\eps}
\,,
\end{align*}
where we have chosen $Y=X^{1/12}$.
Using the values of $\delta$ and $\kappa$ we have available to us,
we are forced to accept the second error term
which has exponent less than $0.598$.
\end{proof}

\section{An upper bound for octic fields with Galois group $Q_8\rtimes S_4$}\label{S:Q8S4}
In this section we treat the fourth and final group appearing in the statement of Theorem~\ref{T:2}. Similar to Lemma~\ref{lemma: vp(N)}, below we collect two facts regarding the discriminant of an $S_4$-quartic field in the case when $G = Q_8 \rtimes S_4$.

\begin{lemma} \label{v_p(D1)}
    Let $G = Q_8 \rtimes S_4$. Write $\Delta_K=D_0 D_1 D_2$ with 
    \[
    (6,D_0)=1, \;\;
    \rad(D_0)\mid \Nm(\Delta_{L/K}), \;\;
    (D_1, 6 \Nm(\Delta_{L/K}))=1, \;\;
    D_2\mid 6^m.
    \]
    Then we have
    \begin{enumerate}
        \item $D_0^{1/3} \leq N_0 \leq D_0^3$,

        \item $D_1 = \rad(D_1)^2$.
    \end{enumerate}
\end{lemma}
\begin{proof}
    (1) Since $v_p(\Delta_K) \leq 3$ and $\text{rad}(D_0) = \text{rad}(N_0)$, the first inequality is immediate. If $v_p(N_0) \in \{1,2,3\}$, the second inequality is proved. To this end, we note that $v_p(N_0) \neq 4$ since the indices of all permutations in $\text{8T40}$ are $\{2,4,8\}$.

    (2) It is enough to show that if $v_p(\Delta_K) \notin \{0,2\}$, then $v_p(\Nm(\Delta_{L/K})) > 0$. If $v_p(\Delta_K) = 3$, then the above lemma gives us the desired inequality. So now assume that $v_p(\Delta_K) = 1$. This implies that there is one ramified prime in $K$ lying over $p$ having inertia degree 1; that is, $\sigma_K(p) \in \{ (1^2 11), (1^2 2)\}$. Suppose towards a contradiction that $v_p(\Nm(\Delta_{L/K})) = 0$. Then $v_p(\Delta_L) = 2$ and the inertia groups $I(\widetilde{\mathfrak{P}}/p)$ are generated by a permutation $\tau$ of index 2 in $8\text{T}40$. All such permutations have order two, so $e(\widetilde{L}/p) = 2$. However, for every such $\langle \tau \rangle$, one can check that for any $D \leq G$ containing $\langle \tau \rangle$ as a normal subgroup, the intersection of $D$ with any $C_2 \times S_4 \cong \text{Gal}(\widetilde{L}/K)$ is nontrivial. We immediately conclude that for any prime ideal $\mathfrak{p}$ in $K$ above $p$, $\mathfrak{p}$ ramifies in $\widetilde{L}/K$. However, this is a contradiction, as $e(\widetilde{L}/p) = 2$ and there is a prime ideal of $K$ above $p$ that is already ramified.
\end{proof}

\begin{proof}[Proof of Theorem~\ref{T:2} for $G=Q_8\rtimes S_4$]
We follow the proof for $\GL_2(\F_3)$ but some details are different. We note that in the subsequent calculations, we use the following analog of Lemma~\ref{L:fourth}, which is proved in the exact same way, except where $v_p(N_1) = 4$ is replaced by $v_p(N_1) \geq 2$:
\[
\sum_{
\substack{
    [L:K] = 2\\
    \Gal(\widetilde{L}/\mathbb{Q}) \cong G\\
    N_1 \leq Y\\
}
}{1}
\ll |\Delta_K|^{\kappa + \epsilon}Y^{1/2}.
\]
We write $\Nm(\Delta_{L/K})=N_0 N_1 N_2$
where these quantities are defined the same way as in Section~\ref{S:GL2F3}.
The analog of Lemma~\ref{lemma: vp(N)} is not totally true.
We have $v_p(N_1)\geq 2$ for all $p\mid N_1$, which is slightly weaker than Part 1 of the Lemma.
We condition on the value of $\rad(D_1)$.

First, we consider the case where $\rad(D_1)\leq Y$.
Notice that the condition $|\Delta_L|\leq X$ implies $N_1\leq X/(\Delta_K^2 N_0)$.
Further, this implies
$$
  N_1\leq \frac{X}{\Delta_K^2 D_0^{1/3}}\leq\frac{2^a 3^bXY^{2/3}}{\Delta_K^{2+1/3}}
  \,.
$$
Indeed, observe that $|\Delta_K|=D_0 D_1 D_2\leq 2^a 3^b D_0 Y^2$,
which implies $D_0\geq \Delta_K/(2^a 3^bY^2)$. 

We estimate: 
\begin{align*}
\sum_{
\substack{
|\Delta_K|\leq \sqrt{X}\\
\rad(D_1) \leq Y\\
}
}
\sum_{
\substack{
N_0 N_1\leq X/\Delta_K^2
}
}
1
&\leq
\sum_{|\Delta_K|\leq \sqrt{X}}
\sum_{\substack{N_1\leq\frac{2^a 3^bXY^{2/3}}{\Delta_K^{2+1/3}}}}
1
\\
&\ll
\sum_{|\Delta_K|\leq \sqrt{X}}
|\Delta_K|^{\kappa + \eps}
\left( 
\frac{XY^{2/3}}{\Delta_K^{2 + 1/3}}
\right)^{1/2}
\\
&
=
X^{1/2}Y^{1/3}
\sum_{|\Delta_K|\leq \sqrt{X}}
|\Delta_K|^{\kappa - 7/6 + \eps}
\\
&\ll X^{1/2+\kappa/2-1/12+\eps}Y^{1/3}.
\end{align*}
In the case where $\rad(N_0')>Y$, we have
\begin{align*}
    \sum_{
\substack{
|\Delta_K|\leq \sqrt{X}\\
\rad(D_1) > Y\\
}
}
\sum_{
\substack{
N_0 N_1\leq X/\Delta_K^2
}
}
1
&\leq
\sum_{
\substack{|\Delta_K|\leq \sqrt{X}\\\exists q>Y\,q^2\mid\Delta_K}
}
\sum_{\substack{N_1\leq X/\Delta_K^2\\ }}
1
\\
&\ll
\sum_{
\substack{|\Delta_K|\leq \sqrt{X}\\ \exists q>Y\,q^2\mid\Delta_K}
}
|\Delta_K|^{\kappa + \eps}
\left(
\frac{X}{\Delta_K^2}
\right)^{1/2}
\\[1ex]
&=
X^{1/2}
\sum_{\substack{|\Delta_K|\leq \sqrt{X}\\\exists q>Y\,q^2\mid\Delta_K}}
|\Delta_K|^{\kappa-1+\eps}
\\ &
\ll X^{1/2+\kappa/2-\delta/2+\eps}+X^{1/2+\kappa/2+\eps}Y^{-1}.
\end{align*}
Then we choose $Y=X^{1/16}$ and the final estimate is
\[
\ll X^{1/2+\kappa/2-1/16}+X^{1/2+\kappa/2-\delta/2+\eps}.
\]
As before, we are forced to accept the second error term due to the values of $\delta$, $\kappa$ that we
have available.




\end{proof}

\begin{proof}[Proof of Theorem~\ref{T:1}]
Combine Lemma~\ref{L:1}, Theorem~\ref{T:2}, and Theorem~\ref{T:3}.
\end{proof}

\section{Acknowledgements}
The authors would like to thank Brandon Alberts and Frank Thorne for their helpful comments.
This research was conducted as part of the Research Experience for Undergraduates and
Teachers program at California State University, Chico, supported by NSF grant DMS2244020.

\bibliography{malle}{}
\bibliographystyle{plain}
\end{document}